\documentclass{amsart}

\usepackage{amssymb}
\usepackage[all]{xy}

\def\Z{{\mathbb Z}}
\def\Q{{\mathbb Q}}
\def\C{{\mathbb C}}
\def\P{{\mathbb P}}
\def\L{{\mathbb L}}
\def\cG{{\mathcal G}}

\def\cN{{\mathcal N}}
\def\O{{\mathcal O}}
\def\cP{{\mathcal P}}
\def\U{{\mathcal U}}
\def\X{{\mathcal X}}
\def\cZ{{\mathcal Z}}

\def\w{\omega}

\def\g{{\mathfrak g}}
\def\n{{\mathfrak n}}
\def\p{{\mathfrak p}}
\def\r{{\mathfrak r}}

\def\u{{\mathfrak u}}
\def\z{{\mathfrak z}}

\def\sigmatilde{{\tilde{\sigma}}}
\def\phihat{\hat{\phi}}
\def\etabar{{\overline{\eta}}}
\def\wtilde{{\tilde{\w}}}
\def\what{{\hat{\w}}}
\def\thetatilde{\tilde{\theta}}

\def\Ghat{\widehat{G}}

\def\uhat{\hat{\u}}
\def\ghat{\hat{\g}}

\def\Uhat{\widehat{\U}}
\def\cGhat{\widehat{\cG}}

\def\kbar{{\overline{k}}}
\def\Gbar{{\overline{G}}}
\def\Rbar{{\overline{R}}}
\def\cGbar{{\overline{\cG}}}

\def\sbar{\overline{s}}
\def\vtilde{\tilde{v}}

\def\tinv{t^{-1}}

\def\Ql{{\Q_\ell}}

\def\Gm{{\mathbb{G}_m}}
\def\cts{\mathrm{cts}}
\def\et{\mathrm{\acute{e}t}}
\def\nab{\mathrm{nab}}
\def\dr{\mathrm{DR}}

\def\Het{H_\et}
\def\Hnab{H_\nab}

\def\dot{\bullet}
\def\blank{\phantom{x}}
\def\Rep{{\sf{Rep}}}
\def\Vec{{\sf{Vec}}}

\newcommand\id{\operatorname{id}}
\newcommand\Hom{\operatorname{Hom}}
\newcommand\Ext{\operatorname{Ext}}
\newcommand\Aut{\operatorname{Aut}}
\newcommand\Out{\operatorname{Out}}
\newcommand\Gr{\operatorname{Gr}}

\newcommand\Char{\operatorname{char}}

\newcommand\Sect{\operatorname{Sect}}
\newtheorem{theorem}{Theorem}[section]
\newtheorem{lemma}[theorem]{Lemma}
\newtheorem{proposition}[theorem]{Proposition}
\newtheorem{corollary}[theorem]{Corollary}
\newtheorem{bigtheorem}{Theorem}
\newtheorem{bigcorollary}[bigtheorem]{Corollary}

\theoremstyle{definition}

\theoremstyle{remark}

\begin{document}

\title{Remarks on Non-Abelian Cohomology of Proalgebraic Groups}

\author{Richard Hain}
\address{Department of Mathematics\\ Duke University\\
Durham, NC 27708-0320}
\email{hain@math.duke.edu}

\thanks{Supported in part by grant DMS-1005675 from the National Science
Foundation.}

\date{\today}

%\begin{abstract}
%
%\end{abstract}

\maketitle

\section{Introduction}

In this note we present a version of Kim's non-abelian cohomology
\cite{kim:coho} which appears to be well suited to computations with motivic
Galois groups. Whereas Kim considers extensions of a profinite group by a
unipotent group, we consider certain extensions of a proalgebraic group $\cG$ by
a prounipotent group $\cP$. This allows us to work instead with Lie algebras.
This, combined with the fact that we work in a situation where all of the basic
objects have weight filtrations with good exactness properties, helps make our
version of non-abelian cohomology somewhat computable.

Throughout, $F$ is a field of characteristic zero. Suppose that $R$ is a
reductive, algebraic $F$-group endowed with a non-trivial, central cocharacter
$\w : \Gm \to R$. We suppose that $\cG$ is a negatively weighted extension
$$
1 \to \U \to \cG \to R \to 1
$$
of $R$ by a prounipotent group $\U$ in the sense of
\cite{hain-matsumoto:weighted}. That is, the weights of the natural action
$$
\xymatrix{
\Gm \ar[r]^\w & R & \ar[l]_\simeq \cG/\U \ar[rr]^(.4){\text{conjugation}}
&& \Aut H_1(\U)
}
$$
of $\Gm$ on the abelianization $H_1(\U)$ of $\U$ are all negative. As we shall
recall in Section~\ref{sec:weights}, every pro-$\cG$-module $V$ has a natural
weight filtration $W_\dot V$ with good exactness properties.

The coefficient group is a prounipotent $F$-group $\cP$, which we assume to have
trivial center. Fix an {\em outer} action $\phi : \cG \to \Out \cP$ of $\cG$ on
$\cP$. This induces an action of $\cG$ on $H_1(\cP)$, so that $H_1(\cP)$ has a
natural weight filtration. We will assume that this is negatively weighted ---
i.e., $H_1(\cP)=W_{-1}H_1(\cP)$ --- and that each of its weight graded quotients
is finite dimensional. Under these conditions, the Lie algebra $\p$ of $\cP$ has
a natural weight filtration with finite dimensional quotients. It exponentiates
to a weight filtration of $\cP$. The group of weight filtration preserving
automorphisms
$$
\Aut_W\cP :=\varprojlim_{n<0} \Aut_W (\cP/W_n\cP)
$$
of $\cP$ is a proalgebraic group. Our goal is to define a non-abelian cohomology
scheme $\Hnab^1(\cG,\cP)$, which will represent $\cP$-conjugacy classes of
sections of an extension of $\cG$ by $\cP$.

The condition that $\cP$ have trivial center implies that this extension is
uniquely determined by the outer action $\phi$. Indeed, since $\cP$ has trivial
center, the sequence
$$
1 \to \cP \to \Aut_W \cP \to \Out_W \cP \to 1,
$$
where the first map takes an element of $\cP$ to the corresponding inner
automorphism, is an exact sequence of proalgebraic $F$-groups. Pulling this
extension back along $\phi$ gives an extension
\begin{equation}
\label{eqn:extn-phi}
1 \to \cP \to \cGhat_\phi \to \cG \to 1
\end{equation}
of $\cG$ by $\cP$. We do not make the conventional assumption that this sequence
splits, nor do we distinguish any one section when it does split.

Denote the Lie algebra of $\cG$ by $\g$. It acts on the weight graded quotients
of $\p$. Denote the group of components of $R$ by $\Rbar$. Frequently we will
impose the finiteness condition
\begin{equation}
\label{eqn:finiteness}
\dim H^1(\g,\Gr^W_{-m} \p)^\Rbar < \infty
\end{equation}
where the $\Rbar$-action on the cohomology group is induced by the adjoint
action of $\cGhat$ on $\ghat$ and $\p$. Examples of extensions that satisfy all
of the conditions above, including the finiteness conditions, arise naturally in
the study of sections of families of smooth projective curves over finitely
generated fields, such as those studied in \cite{hain:rat-pts}.

\begin{bigtheorem}[{cf.\ \cite[Prop.~2]{kim:coho}}]
\label{thm:representability}
Suppose that $N\ge 1$ and that the finiteness condition (\ref{eqn:finiteness})
holds for $1\le m < N$. Under the assumptions above, there is an affine
$F$-scheme of finite type $\Hnab^1(\cG,\cP/W_{-N}\cP)$ which represents the
functor that takes an $F$-algebra $A$ to the set
$$
\{\text{sections of }(\cGhat_\phi/W_{-N}\cP)\otimes_F A \to \cG\otimes_F A\}/
\text{conjugation by }\cP(A).
$$
\end{bigtheorem}

The central cocharacter $\w : \Gm \to R$ plays a key role. The theorem is proved
by showing that for each lift $\what : \Gm \to \cGhat_\phi$ of $\w$ to
$\cGhat_\phi$, every $\cP(A)$ orbit of sections of
$(\cGhat_\phi/W_{-N}\cP)\otimes_F A \to \cG\otimes_F A$ contains a unique
$\Gm$-invariant section.

The theorem provides a natural example where the GIT quotient of a scheme by a
unipotent group is a scheme. Two ingredients that contribute to this happy
situation are: (1) there are compatible $\Gm$-actions on the space of sections
and on the unipotent group acting on them, and (2) the weights of the action of
$\Gm$ on the Lie algebra of the unipotent group are all negative. It would be
useful to see if such conditions allow one to construct nice GIT quotients by
non-reductive groups in more general situations.

\begin{bigcorollary}
If the finiteness condition (\ref{eqn:finiteness}) holds for all $m\ge 1$, then
the affine $F$-scheme
$$
\Hnab^1(\cG,\cP) := \varprojlim \Hnab^1(\cG,\cP/W_{-N}\cP)
$$
represents the functor
$$
\{\text{sections of }\cGhat_\phi\otimes_F A \to \cG\otimes_F A\}/
\text{conjugation by }\cP(A).
$$
\end{bigcorollary}

There is an ``exact sequence'' which aids in the computation of
$\Hnab^1(\cG,\cP)$.

\begin{bigtheorem}[{cf.\ \cite[p.~641]{kim:coho}}]
\label{thm:exactness}
Suppose that $N>1$. If the finiteness condition (\ref{eqn:finiteness}) is
satisfied  when $1\le m \le N$, then
\begin{enumerate}

\item there is a morphism
$\delta : \Hnab^1(\cG,\cP/W_{-N}) {\to} H^2(\g,\Gr^W_{-N}\p)^\Rbar$
of $F$-schemes,

\item there is a principal action of $H^1(\g,\Gr^W_{-N}\p)^\Rbar$ on
$\Hnab^1(\cG,\cP/W_{-N-1})$,

\item $\Hnab^1(\cG,\cP/W_{-N})(A)$ is a principal
$H^1(\g,\Gr^W_{-N}\p)(A)^\Rbar$ set over the set of $A$-rational points
$(\delta^{-1}(0))(A)$ of the scheme $\delta^{-1}(0)$ for all $F$-algebras $A$.

\end{enumerate}
\end{bigtheorem}

This ``exact sequence'' is represented by the diagram:
{\small
$$
\xymatrix{
\Hnab^1(\cG,\cP/W_{-N-1}) \save !L(.85)="Hnab";
"Hnab",\ar@(ul,dl)"Hnab"_{{H^1(\g,\Gr^W_{-N}\p)^\Rbar}}\restore
\ar[r]^-(.55)\pi &
\Hnab^1(\cG,\cP/W_{-N}) \ar[r]^(.55){\delta} & H^2(\g,\Gr^W_{-N}\p)^\Rbar
}
$$
}
in which $\pi$ is the projection that takes a section to its quotient by
$\Gr^W_{-N}\p$.

In our version of non-abelian cohomology, we do not need to specify a splitting
of the extension (\ref{eqn:extn-phi}), or even have one at all. Since this
extension may not split, $\Hnab^1(\cG,\cP)$ may be the empty scheme. Even when
the extension does split, there may be no preferred section;  these cohomology
groups are not pointed, as is customary. For this I make no apologies --- the
computations in \cite{hain:rat-pts} have convinced me that preferring any one
section over another is unnatural in many contexts.

\section{Weight Filtrations}
\label{sec:weights}

In this section, we review some basic facts from
\cite[\S3]{hain-matsumoto:weighted} about negatively weighted extensions and
their representations. As in the introduction, $R$ is a reductive $F$-group and
$\w : \Gm \to R$ is a non-trivial central cocharacter. Suppose that
\begin{equation}
\label{eqn:extn}
1 \to \U \to \cG \to R \to 1
\end{equation}
is negatively weighted extension of $R$ by a prounipotent group (in the category
of proalgebraic $F$-groups). Since $\U$ is prounipotent, Levi's Theorem implies
that (\ref{eqn:extn}) splits and that any two splittings are conjugate by an
element of $\U(F)$. Composing $\w$ with any such splitting gives a lift $\wtilde
: \Gm \to \cG$. Any two such splittings are conjugate by an element of $\U(F)$
as can be seen by applying Levi's Theorem to the pullback of (\ref{eqn:extn})
along $\w$.

Suppose that $V$ is a finite dimensional $\cG$-module. Each lift $\wtilde : \Gm
\to \cG$ of $\w$ to $\cG$ determines a splitting
$$
V = \bigoplus_{m\in \Z} V_m
$$
of $V$ into isotypical $\Gm$-modules, where $\Gm$ acts via $\wtilde$ on $V_m$ by
the $m$th power of its standard representation. We will call this decomposition
the {\em $\wtilde$-splitting} of $V$. It is preserved by homomorphisms $V \to
V'$ between finite dimensional $\cG$-modules.

Although the splitting of $V$ depends on $\wtilde$, the {\em weight filtration}
$$
0 = W_M V \subseteq W_{M+1} V \subseteq \cdots
\subseteq W_{N-1}V \subseteq W_N V = V,
$$
which is defined by
$$
W_n V := \bigoplus_{m\le n} V_m,
$$
does not. This is a filtration by $\cG$-submodules. Each choice of $\wtilde$
determines a natural isomorphism
$$
V_m \cong \Gr^W_m V := W_mV/W_{m-1}V.
$$
It follows that, for each $m\in \Z$, the functor $\Gr^W_m$ from the category of
finite dimensional $\cG$-modules to the category of $R$-modules is exact.

The definition of the weight filtration extends to the categories of ind- and
pro-objects of the category of finite dimensional $\cG$-modules. The functors
$\Gr^W_m$ remain exact.

We shall need the following useful property of negatively weighted extensions.
Define the {\em centralizer} of a lift $\wtilde$ of $\w$ to be the
$F$-group that represents the functor
$$
A \mapsto
\{u\in \U(A):u\wtilde(\lambda)=\wtilde(\lambda)u\text{ all }\lambda\in\Gm(A)\},
$$
from $F$-algebras to sets.

\begin{lemma}
\label{lem:triv_cent}
For each lift $\wtilde : \Gm \to \cG$ of $\w$ to $\cG$, the centralizer
$C(\wtilde)$ of $\wtilde$ is trivial.
\end{lemma}

\begin{proof}
Fix a lift $\wtilde$ of $\w$. Use this to split the weight filtration of $\u$,
the Lie algebra of $\U$. If $u\in \U(A)$ and
$u\wtilde(\lambda)=\wtilde(\lambda)u$ for all $\lambda \in \Gm(A)$, then
$$
\log u = \wtilde(\lambda)(\log u)\wtilde(\lambda)^{-1}
$$
for all $\lambda \in \Gm(F)$. This implies that $\log u \in \u_0$, the weight 0
factor $\u_0$ of the $\wtilde$-decomposition of $\u$. But since $\cG$ is a
negatively weighted extension of $R$, $\Gr^W_0 \u = 0$, from which it follows
that $\log u=0$ and that $u=1$.
\end{proof}

\section{Cohomology}

\subsection{Cohomology of algebraic groups}

We begin with a quick review of the theory of cohomology of affine algebraic
groups in characteristic zero. The basic results are due to Hochschild
\cite{hochschild} and Hochschild and Serre \cite{hochschild-serre}. More recent
references include Jantzen's book \cite{jantzen}. All algebraic groups
considered will be reduced.

Suppose that $G$ is an affine algebraic group over a field $F$ of characteristic
zero. Denote the Lie algebra of $G$ by $\g$ and its enveloping algebra by $U\g$.
Denote the category of rational representations of $G$ by $\Rep(G)$, and the
category of $U\g$-modules by $\Rep(U\g)$. Every $G$-module is a direct limit of
its finite dimensional submodules. For objects $V$ of $\Rep(G)$ and $M$ of
$\Rep(U\g)$, define
$$
H^\dot(G,V) = \Ext_{\Rep(G)}^\dot(F,V) \text{ and }
H^\dot(\g,M) = \Ext_{\Rep(U\g)}^\dot(F,M),
$$
where $F$ is regarded as a trivial-module. Note that if $G$ is reductive, then
$H^\dot(G,V)$ vanishes in positive degrees and is the set of invariant vectors
$V^G$ in degree $0$.

Since the category $\Rep(G)$ has enough injectives
(\cite{hochschild},\cite[p.~43]{jantzen}), and since $\Rep(U\g)$ has enough
projectives, one can approach the computation of the cohomology of $G$ and $\g$
using standard homological methods.

The natural functor $\Rep(G) \to \Rep(U\g)$ induces a homomorphism $H^\dot(G,V)
\to H^\dot(\g,V)$. It is not, in general, an isomorphism. The difference between
the two is controlled by the de~Rham cohomology of $G$, whose definition we now
recall. The de~Rham complex of $G$ is
$$
\Omega^\dot(G) \cong \Hom_F(\Lambda^\dot\g,\O(G)),
$$
where $\O(G)$ denotes the coordinate ring of $G$. It has a natural differential
$d$. The de~Rham cohomology of $G$ is defined by
$$
H_\dr^\dot(G) := H^\dot(\Omega^\dot(G),d).
$$
When $F$ is a subfield of $\C$, Grothendieck's algebraic de~Rham theorem
\cite{grothendieck} implies that $H^\dot_\dr(G)\otimes_F\C$ is isomorphic to the
singular cohomology (with complex coefficients) of the corresponding complex Lie
group.

\begin{theorem}[{Hochschild \cite[Thm.~5.3]{hochschild}}]
If $G$ is a connected affine algebraic group, then for all rational $G$-modules
$V$, there is an $F$-algebra isomorphism
$$
H^k(\g,V) \cong \bigoplus_{i+j=k} H^i_\dr(G)\otimes_F H^j(G,V)
$$
for all $k\ge 0$. In particular, the natural mapping $H^\dot(G,V) \to
H^\dot(\g,V)$ is injective.
\end{theorem}

There are several notable special cases.

\begin{corollary}
If $G$ is a connected reductive group, then for all rational $G$-modules
$$
H^\dot(\g,V) \cong H^\dot_\dr(G)\otimes V^G.
$$
In particular, if $V$ does not contain the trivial representation, then
$H^\dot(\g,V)=0$.
\end{corollary}

Since the de~Rham cohomology of a unipotent group vanishes, we have:

\begin{corollary}
\label{cor:unipotent}
If $G$ is unipotent, then for all rational $G$-modules $V$, there is a natural
isomorphism $H^\dot(G,V) \cong H^\dot(\g,V)$.
\end{corollary}

Hochschild's theorem is clearly false when $G$ is not connected. To understand
the general case, we need the spectral sequence associated to a group extension.

Suppose that $N$ is a normal subgroup of $G$. Taking a $G$-module $V$ to its
$N$-invariants $V^N$ defines a functor $Q:\Rep(G) \to \Rep(G/N)$. It takes
injectives to injectives. Denote the category of $F$-vector spaces by $\Vec_F$.
Applying the Grothendieck spectral sequence to the diagram of functors
$$
\xymatrix{
\Rep(G) \ar[r]_(.45)Q \ar@/^1.5pc/[rrr]^{\Hom_G(F,\blank)} &
\Rep(G/N) \ar[rr]_(.55){\Hom_{G/N}(F,\blank)} && \Vec_F
}
$$
gives the spectral sequence for a group extension:

\begin{proposition}
If $N$ is a normal subgroup of $G$, then for all $G$-modules $V$, there is a
spectral sequence satisfying
$$
E^{s,t}_2 = H^s(G/N,H^t(K,V)) \Longrightarrow H^{s+t}(G,V).
$$
\end{proposition}

Denote the identity component of $G$ by $G_o$ and the group of components
$G/G_o$ of $G$ by $\Gbar$. The conjugation action of $G$ on itself induces an
action on each of its cohomology groups $H^\dot(G,V)$. A standard argument shows
that this action is trivial. Likewise, the natural action of $\g$ on
$H^\dot(\g,M)$ is trivial, so that the adjoint action of $G_o$ on $H^\dot(\g,V)$
is also trivial. The adjoint action of $G$ on $H^\dot(\g,V)$ thus factors
through a $\Gbar$-action.

\begin{corollary}
For all $G$-modules $V$, the inclusion $G_o \hookrightarrow G$ induces an
$F$-algebra isomorphism $H^\dot(G,V) \cong H^\dot(G_o,V)^\Gbar$ and an inclusion
$H^\dot(G,V) \hookrightarrow H^\dot(\g,V)^\Gbar$.
\end{corollary}

The relationship between Lie algebra cohomology and group cohomology is closest
in degrees $\le 1$.

\begin{proposition}
\label{prop:h1}
If $V^{G_o} = 0$, then $H^1(G,V) = H^1(\g,V)^\Gbar$.
\end{proposition}

\begin{proof}
This is just a combination of the results above. The fact that $V^{G_o}=0$
implies that $H^1(\g,V) = H^0_\dr(G_o)\otimes H^1(G_o,V) = H^1(G_o,V)$, so that
$$
H^1(G,V) = H^1(G_o,V)^\Gbar = H^1(\g,V)^\Gbar.
$$
\end{proof}

Write $G$ as an extension
\begin{equation}
\label{eqn:gp_extn}
1 \to U \to G \to R \to 1
\end{equation}
where $R$ is reductive and $U$ is unipotent. Since $H^j(R,V)$ vanishes in
positive degrees, Corollary~\ref{cor:unipotent} and the spectral sequence of
the extension (\ref{eqn:gp_extn}) imply:

\begin{corollary}
\label{cor:gp_vs_lie}
For all $G$-modules $V$, there are natural isomorphisms
$$
H^\dot(G,V) \cong H^\dot(U,V)^R \cong H^\dot(\u,V)^R.
$$
\end{corollary}

Now suppose that $\w : \Gm \to R$ is a central cocharacter.  Since unipotent
groups are connected, the group $\Rbar$ of components of $R$ is isomorphic to
$\Gbar$.

\begin{corollary}
If (\ref{eqn:gp_extn}) is negatively weighted with respect to $\w$ and if $V$ is
an $R$-module of weight $\neq 0$, then $H^1(G,V) \cong H^1(\g,V)^\Rbar$. This
vanishes when all weights of $V$ are non-negative.
\end{corollary}

\subsection{Cohomology of proalgebraic groups}

Cohomology of pro-algebraic groups and pro-Lie algebras will be continuous
cohomology. More precisely, if $\cG = \varprojlim_\alpha G_\alpha$ is a
proalgebraic group, where each $G_\alpha$ is algebraic, and if $V$ is a finite
dimensional $G_\alpha$-module for $\alpha \ge \alpha_o$, then $H^\dot(\cG,V)$
is the ind-vector space
$$
H^\dot(\cG,V) := \varinjlim_{\alpha \ge \alpha_o} H^\dot(G_\alpha,V).
$$
Similarly with Lie algebra cohomology: if $\g = \varprojlim_\alpha \g_\alpha$,
where each $\g_\alpha$ is finite dimensional, and if $V$ is a finite dimensional
$\g_\alpha$-module for all $\alpha \ge \alpha_o$, then $H^\dot(\g,V)$ is the
ind-vector space
$$
H^\dot(\g,V) := \varinjlim_{\alpha \ge \alpha_o} H^\dot(\g_\alpha,V).
$$

The identity component $\cG_o$ of $\cG$ is the inverse limit of the identity
components of the $G_\alpha$. The group of components of a proalgebraic group is
an inverse limit of finite group schemes. All proalgebraic groups in this paper
will have a finite number of components.

Denote the Lie algebra of $G_\alpha$ by $\g_\alpha$. The Lie algebra $\g$ of
$\cG = \varprojlim G_\alpha$ is defined to be $\varprojlim \g_\alpha$. The
adjoint action of $\cG$ on $\g$ induces a $\cG$-action on $H^\dot(\g,V)$.
Proposition~\ref{prop:h1} implies:

\begin{lemma}
\label{lem:propgp-proplie}
Suppose that $\cG$ is a proalgebraic group whose group $\cGbar$ of components is
finite. If $V$ is a finite dimensional $\cG$-module, then the algebra
homomorphism
$$
H^\dot(\cG,V) \to H^\dot(\g,V)^\cGbar
$$
is injective in all degrees. If $V^{\cG_o}=0$, then this map induces natural
isomorphisms
$$
H^1(\cG,V) \cong H^1(\g,V)^\cGbar.
$$ 
\end{lemma}

Lie algebra cohomology will be computed using the standard Chevalley-Eilenberg
complex. In degree $j$, this is
$$
C^j(\g,V) := \Hom_\cts(\Lambda^j\g,V)
:= \varinjlim_\alpha \Hom(\Lambda^j\g_\alpha,V).
$$
Denote the $j$-cycles and $j$-boundaries by $Z^j(\g,V)$ and $B^j(\g,V)$,
respectively.
Now suppose that $\cG$ is an extension
\begin{equation}
\label{eqn:pro_extn}
1 \to \U \to \cG \to R \to 1
\end{equation}
of a reductive algebraic group $R$ by a prounipotent group $\U$. Denote the Lie
algebra of $\U$ by $\u$ and its abelianization by $H_1(\u)$. The group $\cGbar$
of components of $\cG$ is isomorphic to the group of components $\Rbar$ of $R$,
which is finite. The action of $\cG$ on $H^\dot(\g,V)$ induced by the adjoint
action induces an action of $\Rbar$ on $H^\dot(\g,V)$.

\begin{lemma}
If $V$ is an $R$-module satisfying $V^{R_o}=0$, then there are natural
isomorphisms
$$
H^1(\cG,V) \cong H^1(\g,V)^\Rbar \cong \Hom_R(H_1(\u),V).
$$
\end{lemma}

\begin{proof}
This follows directly from Corollary~\ref{cor:gp_vs_lie},
Lemma~\ref{lem:propgp-proplie} and the fact that $H^1(\u)$ is the continuous
dual of $H_1(\u)$ so that $H^1(\g,V) \cong \Hom_R^\cts(H_1(\u),V)$.
\end{proof}

Now suppose that the extension (\ref{eqn:pro_extn}) is negatively weighted with
respect to a cocharacter $\w : \Gm \to R$. Fix a lift $\wtilde : \Gm \to \cG$ of
$\w$ to $\cG$. The Lie algebras $\g$ of $\cG$ and $\u$ of $\U$ are projective
limits of $\cG$-modules via the adjoint action. Their $\wtilde$-gradings are
preserved  by the bracket. Since $\cG$ is a negatively weighted extension of
$R$, $\g_m = 0$ when $m>0$, $\g_0 \cong \r$, the Lie algebra of $R$, and $\g_m =
\u_m$ when $m<0$. The $\wtilde$-splittings of $V$ and $\g$ pass to the cochains:
$$
C^j(\g,V) = \bigoplus_{m\in \Z}C^j_{(m)}(\g,V).
$$
Since the bracket of $\g$ and the action $\g\otimes V \to V$ preserve the
splittings, the differential of $C^\dot(\g,V)$ also preserves it. Consequently,
the $\wtilde$-splitting passes to cohomology. Denote the weight $m$ summand of
the cohomology group $H^j(\g,V)$ by $H^j_{(m)}(\g,V)$. It is the cohomology of
the complex $C^\dot_{(m)}(\g,V)$.

\begin{lemma}
\label{lem:invar_coho}
If $V$ is a $\cG$-module, then $H^j(\g,V) = H^j_{(0)}(\g,V)$.
\end{lemma}

\begin{proof}
The identity component $\cG_o$ of $\cG$ acts on the chain complex $C^\dot(\g,V)$
via the adjoint action. It therefore acts on its cohomology $H^\dot(\g,V)$.
Since this action is trivial, $\wtilde$ acts trivially on $H^\dot(\g,V)$.
\end{proof}

\section{Proofs of the Main Results}

\subsection{Setup}

The main results follow directly from results proved at the end of this section.
Here we consider an extension
\begin{equation}
\label{eqn:big_ext}
1 \to \cN \to \cGhat \to \cG \to 1
\end{equation}
of affine proalgebraic $F$-groups, where $\cN$ is unipotent and where $\cG$ is a
negatively weighted extension of a reductive group $R$. Then $\cGhat$ is an
extension of $R$ by a prounipotent group $\Uhat$, which is an extension
$$
1 \to \cN \to \Uhat \to \U \to 1.
$$
We will assume that $\cGhat$ is a negatively weighted extension of $R$ by
$\Uhat$. This implies that the Lie algebra $\n$ of $\cN$ has a natural weight
filtration satisfying $\n = W_{-1}\n$. Later, $\cN$ will be $\cP/W_{-N-1}$ and
$\cGhat$ will be the quotient of $\cGhat_\phi$ by $W_{-N-1}\cP$.

Denote the Lie algebras of $\cG$, $\cGhat$, $\U$ and $\Uhat$ by $\g$, $\ghat$,
$\u$, and $\uhat$, respectively.

\subsection{Sections}
Suppose that $A$ is an $F$-algebra. A section
$s$ of $\cGhat\otimes_F A\to \cG\otimes_F A$ induces a Lie algebra homomorphism
$ds : \ghat\otimes_F A \to \g\otimes_F A$.

\begin{lemma}
For all $F$-algebras $A$, the derivative map
$$
\{\text{sections of }\cGhat\otimes_F A \to \cG\otimes_F A\}
\to
\{\cGhat\text{-invariant sections of }\ghat\otimes_F A \to \g\otimes_F A\}
$$
is a bijection.
\end{lemma}

\begin{proof}
Choose a splitting of the homomorphism $\cGhat \to R$. This induces an action of
$R$ on $\Uhat$ and an an isomorphism $\cGhat \cong R\ltimes \Uhat$ of
proalgebraic $F$-groups. The corresponding section of $\cG\to R$ induces a
compatible action of $R$ on $\U$ and an isomorphism $R\ltimes \U$. Sections of
$\cGhat \to \cG$ correspond to $R$-invariant sections of $\Uhat \to \U$. Since
$\Char F = 0$, the logarithm and exponential maps
$$ 
\xymatrix{
\uhat\otimes_F A \ar[r]_\exp & \Uhat(A)\ar@/_/[l]_\log &&
\u\otimes_F A \ar[r]_\exp & \U(A) \ar@/_/[l]_\log
}
$$
are bijections. It follows each of the maps
\begin{align*}
\{\text{sections of }\cGhat\otimes_F A \to \cG\otimes_F A\}
&\to
\{R\text{-invariant sections of }\Uhat\otimes_F A \to \U\otimes_F A\}
\cr
&\to
\{R\text{-invariant sections of }\uhat\otimes_F A \to \u\otimes_F A\}
\cr
&\to
\{\cGhat\text{-invariant sections of }\ghat\otimes_F A \to \g\otimes_F A\}
\end{align*}
is a bijection.
\end{proof}

Let $\what : \Gm \to \cGhat$ be a lift of the cocharacter $\w$. Let $\wtilde :
\Gm \to \cG$ be its composition with the projection $\cGhat \to \cG$. Denote the
composition of $\what$ with the projection $\cGhat \to \cG$ by $\wtilde : \Gm
\to \cG$. Define a section $s$ of $\cGhat\otimes_F A \to \cG\otimes_F A$ to be
{\em $\what$-graded} if $s\circ \wtilde = \what$. Equivalently, a section $s$ is
$\what$-graded if it is $\cGhat$-invariant and its derivative $ds$ is a
$\what$-graded section of $\ghat\otimes_F A \to \g\otimes_F A$ in the sense that
$ds$ commutes with the $\Gm$-actions on $\g$ and $\ghat$ induced by $\what$. The
following observation is key.

\begin{lemma}
Each $\cN$-orbit of graded sections of $\cGhat\otimes_F A \to \cG\otimes_F A$
contains a unique $\what$-graded section.
\end{lemma}

\begin{proof}
Suppose that $s$ is a section of $\cGhat\otimes_F A \to \cG_F \otimes A$. Then
one has the lift $s\circ \wtilde$ of $\wtilde$ to $\cGhat$.
Corollary~\ref{cor:appendix} implies that it is conjugate to $\what$ by an
element $u$ of $\cN(A)$:
$$
s\circ \wtilde = u\what u^{-1}.
$$
But then
$$
(u^{-1}su)\circ \wtilde = u^{-1}(s\circ \wtilde)u = \what
$$
so that $u^{-1}su$ is $\what$-invariant. Thus every $\cN(A)$-orbit of sections
contains a $\what$-invariant section.

To prove uniqueness, suppose that $u\in \cN(A)$ and that $s$ and $usu^{-1}$
are both $\what$-invariant. Then $u$ centralizes $\what$ as
$$
\what = (usu^{-1})\circ \wtilde = u(s\circ \wtilde)u^{-1} = u\what u^{-1}.
$$
Lemma~\ref{lem:triv_cent} implies that $u=1$.
\end{proof}

The second key ingredient in the construction of non-abelian cohomology of
negatively weighted extensions is the following result. Note that the adjoint
action of $\cGhat$ on $\ghat$ induces $R$-actions on $\Gr^W_\dot\ghat$ and
$\Gr^W_\dot\g$.

\begin{proposition}
\label{prop:scheme}
There is an ind-affine $F$-scheme $\Sect_\dot^R(\Gr^W_\dot\g,\Gr^W_\dot\n)$
which represents the functor that takes an $F$-algebra $A$ to the set
$$
\{R\text{-invariant graded sections of }
\Gr^W_\dot\ghat\otimes_F A \to \Gr^W_\dot\g\otimes_F A\}.
$$
It is of finite type when $H^1(\g,\Gr^W_m\n)^\Rbar$ is finite dimensional for
all $m\in\Z$.
\end{proposition}

Note that sections of $\Gr^W_\dot\ghat \to \Gr^W_\dot\g$ are $\r$-invariant, and
therefore $R_o$-invariant.

\begin{proof}
Since $\cG$ and $\cGhat$ are negatively weighted extensions of $R$, we have
$$
\Gr^W_0 \g = \Gr^W_0 \ghat = \r.
$$
Thus graded sections of $\Gr^W_\dot \ghat \to \Gr^W_\dot \g$ correspond to
$R$-invariant graded sections of $\Gr^W_\dot \uhat \to \Gr^W_\dot \u$.

Choose an $R$-invariant splitting $\theta$ of $\Gr^W_\dot \u \to \Gr^W_\dot
H_1(\u)$. Since $\u$ is prounipotent, this induces a graded, $R$-invariant
surjection
$$
\thetatilde : \L(\Gr^W_\dot H_1(\u)) \to \Gr^W_\dot\u
$$
from the free Lie algebra generated by $\Gr^W_\dot H_1(\u)$. The $R$-invariant
lifts $s$ of $\thetatilde$
$$
\xymatrix{
& \L(\Gr^W_\dot H_1(\u)) \ar[d]^\thetatilde\ar[dl]_s \cr
\Gr^W_\dot \ghat \ar[r] & \Gr^W_\dot \g
}
$$
are induced by graded $R$-invariant lifts $\Gr^W_\dot H_1(\u) \to \Gr^W_\dot
\ghat$ of $\theta$. The set of $s$ that make the diagram commute comprise an
ind-affine space that is a principal homogeneous space over the ind-vector space
$$
\Hom_R(\Gr^W_\dot H_1(\u),\Gr^W_\dot\n) =
\varinjlim_\alpha\Hom_R(\Gr^W_\dot H_1(\u_\alpha),\Gr^W_\dot\n).
$$
Those $s$ that descend to sections of $\Gr^W_\dot\ghat \to \Gr^W_\dot\g$
are precisely those that satisfy the system of equations
$$
s([x,y]) - [s(x),s(y)] = 0,\qquad x,y\in \Gr^W_\dot \g.
$$
It follows that the the functor $A\mapsto \{\text{graded sections of }
\Gr^W_\dot\ghat\otimes_F A \to \Gr^W_\dot\g\otimes_F A\}$ is represented by an
ind-affine $F$-scheme.

Since
$$
H^1(\g,\Gr^W_m\n)^\Rbar =
\Hom_R(H_1(\u),\Gr^W_m\n) = \Hom_R(\Gr^W_m H_1(\u),\Gr^W_m\n),
$$
the affine space of $R$-invariant lifts $s$ of $\thetatilde$ is of finite type
when $H^1(\g,\Gr^W_m\n)^\Rbar$ is finite dimensional for all $m\in \Z$.
Consequently, the subscheme $\Sect_\dot^R(\Gr^W_\dot\g,\Gr^W_\dot\n)$ is of
finite type when $H^1(\g,\Gr^W_m\n)^\Rbar$ is finite dimensional for all $m\in
\Z$.
\end{proof}

Suppose that $A$ is an $F$-algebra. Each section $s$ of $\cGhat\otimes_F A \to
\cG\otimes_F A$ induces a $\cGhat$-invariant section $ds$ of the Lie algebra
homomorphism $\ghat\otimes_F A \to \g\otimes_F A$ that preserves the weight
filtration. This, in turn, induces an $R$-invariant section $\Gr^W_\dot ds$ of
$\Gr^W_\dot \ghat \otimes_F A\to \Gr^W_\dot \g \otimes_F A$. Observe that, for
each choice $\what$ of a lift of $\w$ to $\cGhat$, the composition of the maps
\begin{align}
\label{eqn:maps_sects}
%&\phantom{\cong}\{\what\text{-graded sections of }
%\ghat\otimes_F A \to \g\otimes_F A\} \cr
%&\cong
&\{\what\text{-graded sections of } \cGhat\otimes_F A \to \cG\otimes_F A\} \cr
&\hookrightarrow
\{\text{sections of } \cGhat\otimes_F A \to \cG\otimes_F A\} \cr
&\overset{\Gr^W_\dot\! d}{\to}
\{R\text{-invariant graded sections of }
\Gr^W_\dot\ghat\otimes_F A \to \Gr^W_\dot\g\otimes_F A\}
\end{align}
is a bijection.

\begin{corollary}
\label{cor:scheme}
For each lift $\what : \Gm \to \cGhat$ of the central cocharacter $\w: \Gm \to
R$ to $\cGhat$, there is an ind-affine $F$-scheme $\Sect^\what_\dot(\cG,\cN)$
which represents the functor that takes an $F$-algebra $A$ to the set
$$
\{\what\text{-graded sections of } \cGhat\otimes_F A \to \cG\otimes_F A\}
$$
It is of finite type when $H^1(\g,\Gr^W_m\n)^\Rbar$ is finite dimensional for
all $m\in\Z$.
\end{corollary}

\begin{corollary}
\label{cor:principal}
For all $F$-algebras $A$ and all lifts $\what$ of $\w$ to $\cGhat$, the function
$$
\cN(A) \times \Sect^\what_\dot(\cG,\cN)(A) \to
\{\text{sections of } \cGhat\otimes_F A \to \cG\otimes_F A\}
$$
defined by $(u,s)\mapsto usu^{-1}$ is a bijection. Consequently, this functor
from $F$-algebras to sets is represented by the ind-affine $F$-scheme
$\Sect(\cG,\cN) := \cN\times\Sect^\what_\dot(\cG,\cN)$. It has finite type when
$H^1(\g,\Gr^W_m\n)^\Rbar$ is finite dimensional for all $m\in\Z$.
\end{corollary}

Yoneda's Lemma implies that the scheme $\Sect(\cG,\cN)$ is independent of the
choice of the lift $\what$ of $\w$.

\subsection{Existence of non-abelian cohomology}

Theorem~\ref{thm:representability} follows directly from:

\begin{theorem}
There is an ind-affine scheme $\Hnab^1(\cG,\cN)$ that represents the functor
that takes an $F$-algebra $A$ to the set
$$
\{\text{sections of } \cGhat\otimes_F A \to \cG\otimes_F A\}/\cN(A)
$$
where $\cN(A)$ acts on sections by conjugation. It is of finite type when
$H^1(\g,\Gr^W_m\n)^\Rbar$ is finite dimensional for all $m\in \Z$.
\end{theorem}

\begin{proof}
The composite of the maps (\ref{eqn:maps_sects}) is a bijection for each choice
of a lift $\what$ of $\w$. Corollary~\ref{cor:principal} implies that for all
$F$-algebras $A$, the natural maps
\begin{multline*}
\Sect^\what_\dot(\cG,\cN)(A)
\to \{\text{sections of } \cGhat\otimes_F A \to \cG\otimes_F A\}/\cN(A) \cr
\to \Sect_\dot^R(\Gr^W_\dot\g,\Gr^W_\dot \n)(A)
\end{multline*}
are bijections. The result now follows from Proposition~\ref{prop:scheme} and
Corollary~\ref{cor:scheme}.
\end{proof}

A corollary of the proof is the following useful description of
$\Hnab^1(\cG,\cN)$.

\begin{corollary}
\label{cor:concrete}
The morphism of ind-schemes
$$
\Sect(\cG,\cN) \to \Sect_\dot^R(\Gr^W_\dot\g,\Gr^W_\dot\n)
$$
that takes a section $s$ to $\Gr^W_\dot ds$ induces an
isomorphism
$$
\Hnab^1(\cG,\cN) \cong \Sect_\dot^R(\Gr^W_\dot\g,\Gr^W_\dot\n).
$$
\end{corollary}

\subsection{The exact sequence}
We continue with the notation of the previous two sections. Suppose that
$W_{-N-1}\n = 0$. Set $\cZ = W_{-N}\cN$, which we suppose to be non-zero. Denote
the Lie algebra of $\cZ$ by $\z$. Theorem~\ref{thm:exactness} is a direct
consequence of:

\begin{proposition}
\label{prop:nab_special}
Assume that $H^1(\g,\Gr^W_{-m}\n)^\Rbar$ is finite dimensional for all $m\in \Z$
so that $\Hnab^1(\cG,\cN)$ and $\Hnab^1(\cG,\cN/\cZ)$ are affine $F$-schemes of
finite type. Then
\begin{enumerate}

\item there is a morphism $\delta : \Hnab^1(\cG,\cN/\cZ) {\to} H^2(\g,\z)^\Rbar$
of $F$-schemes,

\item there is a principal action of $H^1(\g,\z)^\Rbar$ on $\Hnab^1(\cG,\cN)$,

\item for all $F$-algebras $A$, $\Hnab^1(\cG,\cN)(A)$ is a principal
$H^1(\g,\z)(A)^\Rbar$-set over the $A$-rational points $(\delta^{-1}(0))(A)$ of
the scheme $\delta^{-1}(0)$.

\end{enumerate}
\end{proposition}

Informally, we summarize this by saying that the sequence
$$
\xymatrix{
\Hnab^1(\cG,\cN) \save !L(.85)="Hnab";
"Hnab",\ar@(ul,dl)"Hnab"_{{H^1(\g,\z)^\Rbar}}\restore
\ar[r]^-(.55)\pi &
\Hnab^1(\cG,\cN/\cZ) \ar[r]^(.55){\delta} & H^2(\g,\z)^\Rbar
}
$$
is exact, where $\pi$ is the projection that takes a section to its quotient by
$\cZ$.

\begin{proof}
Denote $\Gr^W_\dot\g$, $\Gr^W_\dot\n$ and $\Gr^W_\dot\ghat$ by $\g_\dot$,
$\n_\dot$ and $\ghat_\dot$, respectively. Identify $\Hnab^1(\cG,\cN)$ with the
scheme of graded, $R$-invariant sections of $\ghat_\dot \to \g_\dot$ and
$\Hnab^1(\cG,\cN/\cZ)$ with the scheme of graded, $R$-invariant sections of
$\ghat_\dot/\z \to \g_\dot$.

We first define the ``connecting morphism'' $\delta$. Suppose that $\sigma$ is a
graded, $R$-invariant section of $(\ghat_\dot/\z)\otimes_F A \to
\g_\dot\otimes_F A$. Define $\sigmatilde$ to be the section
$$
\xymatrix{
\g_\dot \otimes_F A \ar[r]^(.45)\sigma &
(\ghat_\dot/\z)\otimes_F A \ar[r]^(.55){r\otimes A} & \ghat_\dot \otimes_F A
}
$$
of $\ghat\otimes_F A \to \g\otimes_F A$, where $r$ is the unique, $F$-linear,
graded section of $\ghat \to \ghat/\z$. It is $R$-invariant. Note that, in
general, $r$ is not a Lie algebra homomorphism. The obstruction to $\sigmatilde$
being a Lie algebra homomorphism is the $R$-invariant function $h_\sigma :
\Lambda^2 \g \to \z$ defined by
$$
h_\sigma(x,y) := [\sigmatilde(x),\sigmatilde(y)]- \sigmatilde([x,y]).
$$
It is an $R$-invariant (and hence $\what$-invariant) 2-cocycle: $h_\sigma \in
Z^2_{(0)}(\g_\dot,\z)^R\otimes_F A$. Define
$$
\delta : \Hnab^1(\cG,\cN/\cZ)(A) \to H^2(\g,\z)^\Rbar\otimes_F A
$$
by taking $\sigma$ to the class of $h_\sigma$ in
$H^2_{(0)}(\g_\dot,\z)^R\otimes_F A$, which, by Lemma~\ref{lem:invar_coho}, is
isomorphic to $H^2(\g,\z)^\Rbar\otimes_F A$. This is easily seen to be a
morphism of $F$-schemes.

If $\delta(\sigma) \neq 0$, $\sigma$ does not lift to a graded section of
$\ghat_\dot\otimes_F A \to \ghat\otimes_F A$. If $\delta(\sigma) = 0$,
Lemma~\ref{lem:invar_coho} implies that there is an $R$-invariant 1-cochain $v
\in C^1_{(0)}(\g,\z)\otimes_F A$ such that $\delta(v)=h_\delta$. Then
$\sigmatilde + v : \g_\dot\otimes_F A \to \ghat_\dot\otimes_F A$ is a graded Lie
algebra section of $\ghat_\dot\otimes_F A \to \g_\dot\otimes_F A$ that lifts
$\sigma$. This establishes the exactness at $\Hnab^1(\cG,\cN/\cZ)$.

Our last task is to show that the fibers of $\Hnab^1(\cG,\cN) \to
\Hnab^1(\cG,\cN/\cZ)$ are principal $H^1(\g,\z)^\Rbar$ spaces. First observe
that, since $\z$ has negative weight, $C^0_{(0)}(\g,\z) = \z_0 = 0$.
Consequently, there are no invariant 1-coboundaries and
$$
H^1(\g,\z)^\Rbar = H^1_{(0)}(\g,\z)^\Rbar = Z^1_{(0)}(\g,\z)^R.
$$
Next observe that if $\sigma$ is an $R$-invariant, graded section of
$\ghat_\dot\otimes_F A \to \g_\dot\otimes_F A$ and $f : \g_\dot\to \z\otimes_F
A$ is an $R$-invariant, graded function, then $\sigma+f$ is a
graded section of $\ghat_\dot\otimes_F A \to \g_\dot\otimes_F A$ if and only if
$f \in Z^1_{(0)}(\g,\z)^R\otimes_F A$. The last assertion follows.
\end{proof}

\section{Relation to Kim's work}

Kim \cite{kim:program} is interested in $S$-integral points of varieties
(particularly curves) defined over number fields. Here we give a brief
explanation of how his ideas can be expressed naturally in terms of the variant
of non-abelian cohomology developed in this paper.

Suppose that $k$ is a number field and that $X$ is a smooth, geometrically
connected, projective $k$-scheme. Suppose that $S$ is a finite set of primes in
$\O_k$ and that $\X$ is a smooth model of $X$ over the ring of $S$-integers
$\O_{k,S}$ of $k$. Set $G_{k,S}$ be the Galois group of the maximal algebraic
extension of $k$ unramified outside $S$. It acts on the first cohomology of
$X\otimes_k \kbar$:
$$
\rho_\X : G_{k,S} \to \Aut \Het^1(X\otimes_k \kbar,\Ql(1)).
$$
The Zariski closure $R$ of the image of $\rho_\X$ is a reductive $\Ql$-group
that, by a theorem of Deligne (cf.\ \cite[(2.3)]{serre:ladiques}), contains the
scalar matrices. Since $H:=\Het^1(X\otimes_k \kbar,\Ql(1))$ has weight $-1$, we
define $\w : \Gm \to R$ by $\w(t) = t^{-1}\id_H$.

We will regard $\rho_\X$ as a homomorphism $G_{k,S} \to R(\Ql)$. It is Zariski
dense by construction. Define $\cG$ to be the weighted completion of $G_{k,S}$
with respect to $\rho_\X$ and the cocharacter $\w$.\footnote{Definitions can be
found in \cite{hain-matsumoto:weighted}.} This is a proalgebraic $\Ql$-group
which is an extension of $R$ by a prounipotent group.

Fix a geometric point $\etabar$ of $X$. Denote by $\cGhat$ the weighted
completion of $\pi_1(X,\etabar)$ with respect to the homomorphism
$\pi_1(X,\etabar)\to G_{k,S} \to R(\Ql)$ and $\w$. Denote the continuous
unipotent completion of $\pi_1(X\otimes_k\kbar,\etabar)$ over $\Ql$ by $\cP$.
Since $H_1(\cP) \cong H$ is finite dimensional and has weight $-1$, the action
of $\pi_1(X,\etabar)$ on $\pi_1(X\otimes_k\kbar,\etabar)$ by conjugation induces
an action $\phihat$ of $\cGhat$ on $\cP$. 

Right exactness of these proalgebraic completions implies that the sequence
$$
\cP \to \cGhat \to \cG \to 1
$$
is exact. If $\cP$ has trivial center, as it does when $X$ is a hyperbolic
curve, then the composite
$$
\xymatrix{
\cP \ar[r] & \cGhat \ar[r]^(0.4){\phihat} & \Aut_W\cP
}
$$ 
is injective, which implies that the sequence
$$
1 \to \cP \to \cGhat \to \cG \to 1
$$
is exact. The pullback of this sequence along $G_{k,S} \to \cG(\Ql)$ is an exact
sequence
$$
1 \to \cP(\Ql) \to \Ghat_\X \to G_{k,S} \to 1
$$
This is the kind of extension considered by Kim in the non-abelian cohomology
developed in \cite{kim:coho}. The universal mapping property of weighted
completion implies that the natural mapping
\begin{multline*}
\Hnab^1(\cG,\cP)(\Ql)
=\{\text{sections of }\cGhat \to \cG\}/\cP(\Ql) \cr
\to \{\text{sections of }\Ghat_\X \to G_{k,S}\}/\cP(\Ql)
\end{multline*}
is a bijection. This implies that the version of non-abelian cohomology 
developed in \cite{kim:coho} agrees with the one developed in this paper, at
least on $\Ql$-rational points.

\appendix

\section{Cohomology of $\Gm\otimes A$}

This material should be standard. But just in case, we give an exposition. Here
$\Gm$ will be regarded as a group scheme over $\Z$. Denote the rank 1
$\Gm$-module given by the $d$th power of the standard representation by $V_d$.

\begin{proposition}
\label{prop:appendix}
For all $\Z$-algebras $A$ the cohomology group $H^1(\Gm\otimes_\Z A,V_d)$
vanishes.
\end{proposition}

\begin{proof}
The group $H^1(\Gm\otimes_\Z A,V_d)$ is the space of $V_d$-valued cocycles
modulo the $V_d$ valued coboundaries. The 1-cocycles on $\Gm\otimes_\Z A$
correspond to elements $f(t) \in A[t,\tinv]$ that satisfy the cocycle condition
\begin{equation}
\label{eqn:cocycle}
f(t\otimes t) = f(t\otimes 1) + (t^d\otimes 1)f(1\otimes t)
\in A[t,\tinv]\otimes_A A[t,\tinv]
\end{equation}
The 1-coboundaries are all of the form $f(t) = a(t^d-1)$ for some $a\in A$.

To see that every cocycle is a coboundary, write $f(t) = \sum_{-N}^N a_n t^n$,
where each $a_n \in A$. The cocycle condition (\ref{eqn:cocycle}) becomes
$$
\sum a_n t^n \otimes t^n = \sum a_n t^n\otimes 1 + \sum a_n t^d\otimes t^n.
$$
Note that the only ``diagonal terms'' that appear on the right-hand side
are $1\otimes 1$ and possibly $t^d\otimes t^d$.

When $d=0$ the cocycle condition (\ref{eqn:cocycle}) implies that $f(t)=0$,
so that there are no non-zero cocycles and no coboundaries. If $d\neq 0$, then
the cocycle condition implies that $a_0 + a_d = 0$ and that $a_n = 0$ when
$n\neq 0, d$. That is, $f(t) = a_d(t^d - 1)$, which is a coboundary.
\end{proof}

\begin{corollary}
\label{cor:appendix}
Suppose that $A$ is a $\Z$-algebra and that
$$
0 \to U \to G \to \Gm\otimes_\Z A \to 1
$$
is an extension of $\Gm\otimes_\Z A$ in the category of affine $A$-groups, where
$U$ is unipotent. If $U$ has a descending central series
$$
U = U^1 \supseteq U^2 \supseteq U^3 \supseteq \cdots \supseteq U^N \supseteq
U^{N+1} = 1,
$$
where each graded quotient $U^m/U^{m+1}$ is a direct sum of $\Gm\otimes_\Z
A$-modules isomorphic to $V_d$ for some $d$, then any two sections of $G \to
\Gm\otimes_\Z A$ over $A$ are conjugate by an element of $U(A)$.
\end{corollary}

\begin{proof}
When $U$ is abelian, the result follows directly from
Proposition~\ref{prop:appendix} by standard arguments. Suppose that $N>1$ and
that the result is true when the length of the lower central series of the
unipotent kernel is $<N$. Suppose that the length of the lower central series of
$U$ is $N$. Suppose that $s_0$ and $s_1$ are two sections of $G \to
\Gm\otimes_\Z A$. Let $Z$ be the center of $U$. This is normal in $G$, so one
has the extension
$$
1 \to U/Z \to G/Z \to \Gm\otimes_\Z A \to 1.
$$
of affine $A$-groups. The sections $s_0$ and $s_1$ induce sections $\sbar_0$
and $\sbar_1$ of this extension. By our induction hypothesis, there is an
element $v$ of $(U/Z)(A)$ such that $\sbar_1 = v\sbar_0 v^{-1}$. Since $U$ is
unipotent, the sequence
$$
1 \to Z(A) \to U(A) \to (U/Z)(A) \to 1
$$
is exact. Choose a lift $\vtilde \in U(A)$ of $v$. Consider the diagram
$$
\xymatrix{
1 \ar[r] & Z \ar[r]\ar@{=}[d] & G \ar[r] & G/Z \ar[r] & 1\cr
1 \ar[r] & Z \ar[r] & H \ar[r]\ar[u] & \Gm\otimes_\Z A
\ar[r]\ar[u]_{\sbar_1} & 1
}
$$
where the right-hand square is a pullback. The sections $s_1$ and $\vtilde s_0
\vtilde^{-1}$ induce sections $\sigma_1$ and $\sigma_0$ of $H \to \Gm\otimes_\Z
A$. Since $Z$ is abelian, there exists $z\in Z(A)$ such that $\sigma_1=z\sigma_0
z^{-1}$. This implies that $s_1 = (z\vtilde) s_0 (z\vtilde)^{-1}$.
\end{proof}

\end{document}